\documentclass[bibtotoc,a4paper,10pt]{article}
\usepackage[top=1.25in,left=1in,right=1in,bottom=1.25in]{geometry}
\usepackage{latexsym}
\usepackage{enumerate,amsmath,amssymb,dsfont,pifont,amsthm}
\usepackage[dvips]{color}
\usepackage{graphicx}
\usepackage[arrow, matrix, curve]{xy}
\usepackage{setspace}
\usepackage{fancyhdr}
\usepackage{amsthm}
\usepackage{latexsym}
\usepackage{dsfont}
\usepackage{bbm}
\usepackage{amssymb}
\usepackage{amsmath}
\usepackage{graphicx}
\usepackage{mathabx}

\newtheorem{theorem}{Theorem}[section]
\newtheorem{prop}[theorem]{Proposition}
\newtheorem{lem}[theorem]{Lemma}
\newtheorem{cor}[theorem]{Corollary}
\newtheorem{rem}[theorem]{Remark}
\newtheorem{exa}[theorem]{Example}

\def\gl{\buildrel \rm def\over =}

\DeclareMathOperator{\one}{\mathbbm{1}}

\begin{document}

\allowdisplaybreaks

\title{\bfseries Copula Relations in Compound Poisson Processes}

\author{%
    \textsc{Christian Palmes}
    \thanks{Lehrstuhl IV, Fakult\"at f\"ur Mathematik, Technische Universit\"at Dortmund,
              D-44227 Dortmund, Germany,
              \texttt{christian.palmes@math.tu-dortmund.de}}
    }


\maketitle

\begin{abstract}
We investigate in multidimensional compound Poisson processes (CPP) the relation between the dependence structure of the jump distribution and the dependence structure of the respective components of the CPP itself. For this purpose the asymptotic $\lambda t\to \infty$ is considered, where $\lambda$ denotes the intensity and $t$ the time point of the CPP. For modeling the dependence structures we are using the concept of copulas. We prove that the copula of a CPP converges under quite general assumptions to a specific Gaussian copula, depending on the underlying jump distribution.

Let $F$ be a $d$-dimensional jump distribution $(d\geq 2)$, $\lambda>0$ and let $\Psi(\lambda,F)$ be the distribution of the corresponding CPP with intensity $\lambda$ at the time point $1$. Further, denote the operator which maps a $d$-dimensional distribution on its copula as $\mathcal{T}$. The starting point of our investigation was the validity of the equation
\begin{equation} \label{marFreeEq}
\mathcal{T}(\Psi(\lambda,F))=\mathcal{T}(\Psi(\lambda,\mathcal{T}F)).
\end{equation}
Our asymptotic theory implies that this equation is, in general, not true. 

A simulation study that confirms our theoretical results is given in the last section.
\end{abstract}



\section{Introduction}

Let $(N_t)_{t\geq 0}$ be a Poisson process on a probability space $(\Omega,\mathcal{F},P)$ with intensity $\lambda>0$ and let
$$X_j\,:\,(\Omega,\mathcal{F}) \to (\mathds{R},\mathcal{B}(\mathds{R}^d)),\quad j\in\mathds{N}$$
be a sequence of i.i.d. random variables, such that $(X_j)_{j\in\mathds{N}}$ and $(N_t)_{t\geq 0}$ are independent. Set $F\gl P^{X_1}$ and
\begin{equation} \label{CPPEq}
Y_t\gl \sum_{j=1}^{N_t} X_j,\quad t\geq 0.
\end{equation}
$Y$ is a compound Poisson process with the $d$-dimensional jump distribution $F$. Given the $n$ equidistant observations $Y_1,\ldots,Y_n$, Buchmann and Gr\"{u}bel \cite{BucGru} propose in a one dimensional setting a deconvolution based method to estimate the underlying jump distribution $F$. In doing so they investigate relations between the jump distribution and the distribution of the resulting compound Poisson process (CPP). To get more convenient with their approach we citate Lemma 7 in \cite{BucGru}:

\newtheorem*{BucGruLem}{Lemma 7 of Buchmann, Gr\"{u}bel \cite{BucGru}}

\begin{BucGruLem}
Let $F$ and $G$ be probability distributions on $\mathds{R}_+\gl [0,\infty)$ with $\int_{(0,\infty)} e^{-\tau y}\,G(dy) < e^{-\lambda}$ for some $\lambda,\tau>0$ and
$$G\gl \Psi(\lambda,F) \gl e^{-\lambda}\sum_{j=0}^{\infty} \frac{\lambda^j}{j!} F^{\ast j}.$$
Then, it holds
\begin{equation} \label{BanconvEq}
F= \Phi(\lambda,G)\gl \sum_{j=1}^\infty \frac{(-1)^{j+1}e^{\lambda j}}{\lambda j} G^{\ast j}.
\end{equation}
\end{BucGruLem}
The convergence of the right-hand sum in (\ref{BanconvEq}) holds in some suitable Banach space $D(\tau)$ introduced in detail in Buchmann, Gr\"ubel \cite{BucGru}. Note, that we have by considering (\ref{CPPEq}) the relation 
$$P^{Y_t-Y_{t-1}}=\Psi(\lambda,F),\quad t=1,\ldots,n.$$ 
In this sense, $\Phi$ in the cited Lemma 7 regains the jump distribution $F$ out of $\Psi(\lambda,F)$. 

In this paper we analyse in a multidimensional setting the relation between the copula of the jump distribution and the copula of the associated CPP under the asymptotic $\lambda t\to\infty$. Our investigations imply for instance that a multidimensional copula analogue of the above Lemma 7 in \cite{BucGru} does not hold. To be more precise it is even not possible to define a copula analogue version of the function $\Psi$, cf. Remark \ref{copNotFunc}.

\subsection*{Organisation of this paper}
To keep the technical overhead as small as possible, we assume in what follows w.l.o.g. the case $d=2$. 

\textit{Section 2} simply states some useful definitions for our needs. If $F$ is a two dimensional distribution with continuous margins, i.e. $F\in\mathcal{M}^c$, we denote with $\mathcal{T}F$ its unique copula, cf. Proposition \ref{copUniqueProp}.

In \textit{Section 3}, we consider the copula of a compound Poisson process $Y$ under the asymptotic $\lambda t\to\infty$, i.e. we consider the limit behaviour of
\begin{equation} \label{copAsympEq}
\mathcal{T} P^{Y_t} = \mathcal{T} \Psi(\lambda t, P^{X_1}),\quad \lambda t\to\infty.
\end{equation}
Obviously, (\ref{copAsympEq}) implies that we can fix w.l.o.g. $t\gl 1$ and consider only the intensity limit $\lambda\to\infty$. In this context, Theorem \ref{CopGaussTheorem} yields the convergence
$$\mathcal{T}\Psi(\lambda,F) \to \mathcal{T} N(0,\Sigma),\quad \lambda\to\infty$$
which is uniform on $[0,1]^2$. Here, $\Sigma\in\mathds{R}^{2\times 2}$ denotes a positive definite matrix defined in Theorem \ref{CopGaussTheorem}. Thus, it follows that the copula of a CPP converges uniformly under this asymptotic to the Gaussian copula $\mathcal{T} N(0,\Sigma)$. To distinguish between the Gaussian limit copulas, we have to investigate whether $\mathcal{T}N(0,\Sigma) = \mathcal{T}N(0,\Sigma')$ holds for two positive definite matrices $\Sigma,\Sigma'\in\mathds{R}^{2\times 2}$. This is done in Proposition \ref{GaussCopEqProp}: With the notations in Section 2 the entity
$$\rho(P^{X_1}) = \rho(F) \gl \frac{\int xy\,dF(x,y)}{\sqrt{\int x^2\,dF(x,y)\int y^2\,dF(x,y)}}$$
determines the limit copula of the CPP with jump distribution $F=P^{X_1}$. Using this asymptotic approach, the statement of Corollary \ref{copulaNotInjectivCor} implies that (\ref{marFreeEq}) is, in general, not true, compare Remark \ref{copNotFunc}.

In \textit{Section 4}, we analyse the resulting limit copulas of all compound Poisson processes in a certain way. For this purpose, we investigate the map $F\mapsto \rho(F)$. The first interesting question is whether
$$\rho(\mathfrak{C}) {\buildrel ? \over =}  \rho(\mathcal{M}_+^c) = (0,1]$$
holds, where $\mathcal{M}_+^c$ denotes all two dimensional distributions $F$ with $F([0,\infty)^2)=1$ and continuous margins. To say it in prose: The question is, whether the set of jump distributions which consists of the set of copulas $\mathfrak{C}$, can generate every limit copula, which belongs to a CPP with positive jumps. Proposition \ref{CopulaRhoPro} states that this is not the case since we have 
$$\rho(\mathfrak{C})\subseteq \left[\frac{1}{2},1\right].$$
Note, that Cauchy Schwarz yields for every $F$ the inequalities $-1\leq \rho(F) \leq 1$ and that $\rho$ can be geometrically interpreted as the cosinus of the two coordinates of $X_1$ in the Hilbert space $L^2(\Omega,\mathcal{F},P)$. Thus, from the above geometric point of view, copulas always span an angle between $0$ and $60$ degrees. Additionally, 
Example \ref{ClaytonExample} states that all limit copulas that are reachable by a copula jump distribution are even obtained by a Clayton copula, i.e.
$$\rho(\{C_\theta\,:\,\theta\in [-1,\infty)\backslash \{0\}\}\cup \{\Pi\}\cup \{M\}) = \left[\frac{1}{2},1\right],$$
see Section 2 for the notation. Finally, Example \ref{remainderExample} provides the answer to the question of how to obtain the remaining limit copulas which belong to $\rho\in \left(0,\frac{1}{2}\right)$. In this example, we describe a constructive procedure how to construct such jump distributions: Fix any $0<\epsilon<1$. Simulate two independent on $[0,1]$ uniformly distributed, i.e. $U[0,1]$ distributed, random variables $U$ and $V$. If $|U-V|\geq \epsilon$, make the jump $(U,V)\in\mathds{R}^2$. Else repeat this procedure until the difference between $U$ and $V$ is not less that $\epsilon$, and make afterwards the jump $(U,V)\in\mathds{R}^2$. All necessary repetitions are performed independently from each other. Then, if $\epsilon$ runs through the interval $(0,1)$, we get a set of corresponding $\rho$ values that includes $\left(0,\frac{1}{2}\right)$. Observe that the resulting jumps of the above procedure are all positive.

Finally, in \textit{Section 5}, a simulation study is presented. We investigate the finite sample behavior of the convergence statement (\ref{mainConvEq}) in Theorem \ref{CopGaussTheorem} by use of several MATLAB simulations. Clayton copulas are simulated for the jump distribution of the respective compound Poisson processes.

\section{Basic definitions} \label{defSec}

We denote with $\mathcal{M}$ the set of all probability measures on $(\mathds{R}^2,\mathcal{B}^2)$ where $\mathcal{B}^2$ are the Borel sets of $\mathds{R}^2$. Furthermore, define
$$F\in\mathcal{M}^c\,:\iff\,F\in\mathcal{M}\text{ and } 
F(\{x\}\times \mathds{R}) = F(\mathds{R}\times\{x\}) = 0,\quad x\in\mathds{R},$$
i.e. the case that both margins of $F$ are continuous. If we have additionally
$$F([0,x]\times \mathds{R}) = F(\mathds{R}\times [0,x]) = x,\quad 0\leq x\leq 1,$$
we write $F\in\mathfrak{C}$ and call it a copula. Thus, we have defined a further subclass and have altogether the inclusions
$$\mathfrak{C}\subseteq \mathcal{M}^c\subseteq \mathcal{M}.$$
Define next $\mathds{R}_+\gl [0,\infty)$ and
$$\mathcal{M}_+\gl \left\{F\in\mathcal{M}\,:\,F\left(\mathds{R}^2\backslash \mathds{R}_+^2\right)=0\right\},\quad \mathcal{M}_+^c\gl \mathcal{M}_+\cap\mathcal{M}^c$$
and observe $\mathfrak{C}\subseteq \mathcal{M}_+^c$.

For a more convenient notation, we do not distinguish between a probability measure and its distribution function, e.g. we shall write without confusion
$$F((-\infty,x]\times(-\infty,y]) = F(x,y),\quad x,y\in\mathds{R},\,F\in\mathcal{M}.$$
The definition of the map $\mathcal{T}$ in the following proposition is crucial for what follows.
\begin{prop} \label{copUniqueProp}
It exists a unique map
$$\mathcal{T}\,:\,\mathcal{M}^c \to \mathfrak{C}$$
with the property
$$F(x,y) = (\mathcal{T}F)(F_1(x),F_2(y)),\quad x,y\in\mathds{R}$$
where $F_k,\, k=1,2$ denotes the $k$-th marginal distribution of $F\in\mathcal{M}^c$.
\end{prop}
\begin{proof}
This is a consequence of Theorem 2.3.3. (Sklar) in Nelsen \cite{Nel}.
\end{proof}
$\mathcal{T}$ is a map that transforms a probability measure in its copula. We require the margins to be continuous in order to get a unique map. Note that it holds of course $\mathcal{T}_{|\mathfrak{C}} = \text{id}_{|\mathfrak{C}}$. We will deal with the following concrete copulas:
\begin{equation*}
\begin{array} {ll}
\Pi(u,v) = uv & \text{independence copula,} \\
W(u,v) = \max(u+v-1,0) & \text{Fr\'{e}chet-Hoeffding lower bound,} \\
M(u,v) = \min(u,v) & \text{Fr\'{e}chet-Hoeffding upper bound,} \\
C_\theta(u,v) = \left\{\left(\max\{u^{-\theta}+v^{-\theta}-1,0\}\right)^{-\frac{1}{\theta}}\right\}_{\theta\in [-1,\infty)\backslash\{0\}} & \text{family of Clayton copulas,}
\end{array}
\end{equation*}
with $0\leq u,v \leq 1$. Furthermore, define a function $\rho$ via
\begin{equation} \label{rhodef}
\rho(F) = \frac{\int xy\,dF(x,y)}{\sqrt{\int x^2\,dF(x,y){\int y^2\,dF(x,y)}}}
\end{equation}
on the domain of all $F\in\mathcal{M}^c$ which possess square integrable margins, and write as in Buchmann and Gr\"{u}bel \cite{BucGru}
$$\Psi(\lambda,F) = e^{-\lambda}\sum_{k=0}^\infty \frac{\lambda^k}{k!}F^{\ast k},\quad F\in\mathcal{M}.$$

\section{Asymptotic results}

\begin{lem} \label{contLem}
Given $F,\,G\in\mathcal{M}^c$ and $\lambda>0$. Then, it holds $F\ast G\in\mathcal{M}^c$ and $\Psi(\lambda,F) \in \mathcal{M}^c$.
\end{lem}
\begin{proof}
Fubinis theorem yields with any fixed $r\in\mathds{R}$
\begin{eqnarray*}
(F\ast G)_j(\{r\}) &=& (F_j\otimes G_j)(\{(x,y)\in\mathds{R}^2\,:\, x+y =r\}),\quad j=1,2 \\&=&
\int_\mathds{R} \int_\mathds{R} \one_{\{r-y\}}(x)\,dF_j(x)\,dG_j(y) \\ &=&
\int_\mathds{R} F_j(\{r-y\})\,dG_j(y) \\&=&
0.
\end{eqnarray*}

In order to prove the second assertion note that we have
\begin{eqnarray*}
\Psi(\lambda,F)_j(\{r\}) &=& \left(e^{-\lambda} \sum_{k=0}^\infty \frac{\lambda^k}{k!} F^{\ast k}\right)_j(\{r\}),\quad j=1,2 \\ &=&
e^{-\lambda} \sum_{k=0}^\infty \frac{\lambda^k}{k!} (F^{\ast k})_j(\{r\}) \\&=&
0
\end{eqnarray*}
where the last expression is zero because of what we have proven at the beginning.
\end{proof}

\begin{prop} \label{WeakCopConvPro}
Let $(F_n)_{n\in\mathds{N}_0} \subseteq \mathcal{M}^c$ and
$$F_n \,\, {\buildrel \text{d} \over \to}\,\, F_0,\quad n\to\infty.$$
Then, we have
\begin{equation} \label{weakCopClaim}
\sup_{u,v\, \in [0,1]}|(\mathcal{T}F_n)(u,v) - (\mathcal{T}F_0)(u,v)| \to 0,\quad n\to\infty.
\end{equation}
\end{prop}
\begin{proof}
Set $C_n\gl \mathcal{T}F_n$ and let $F_n^j,\,j=1,2$ denote the two marginal distributions of $F_n$. Fix $(u,v)\in(0,1)^2$. Then, there exist $x,y\in\mathds{R}$ with
$$F_0^1(x) = u,\quad F_0^2(y) = v.$$
Observe
$$F_0(x,y) = C_0(F_0^1(x),F_0^2(y)) = C_0(u,v).$$
Since with the margins of $F_0$, $F_0$ itself is also continuous, the assumption $F_n\,\,{\buildrel \text{d} \over \to}\,\,F_0$ yields
\begin{equation} \label{weakCopeq1}
C_n(F_n^1(x),F_n^2(y)) = F_n(x,y) \to F_0(x,y) = C_0(u,v),\quad n\to\infty.
\end{equation}
Next, it holds
\begin{equation} \label{weakCopeq2}
|C_n(u,v)-C_n(F_n^1(x),F_n^2(y))|\leq |u-F_n^1(x)| + |v-F_n^2(y)| \to 0,\quad n\to\infty
\end{equation}
because every copula is Lipschitz continuous, cf. Nelsen \cite{Nel}[Theorem 2.2.4]. The latter convergence to zero results from $F_n^j\,\,{\buildrel \text{d} \over \to}\,\,F_0^j,\,j=1,2$ which is a direct consequence of the Cram\'{e}r-Wold Theorem. The pointwise convergence in (\ref{weakCopClaim}) follows from (\ref{weakCopeq1}) and (\ref{weakCopeq2}). 

For the uniform convergence, fix $\epsilon>0$ and choose any $m>\frac{1}{\epsilon},\,m\in\mathds{N}$. Then, we have for all $0\leq u,v \leq 1$ and
$$u_m\gl \frac{\lfloor um\rfloor}{m},\quad v_m\gl\frac{\lfloor vm\rfloor}{m}$$
\begin{eqnarray*}
|C_n(u,v)-C(u,v)| &\leq& |C_n(u,v)-C_n(u_m,v_m)|+|C_n(u_m,v_m)-C(u_m,v_m)| \\ && 
+|C(u_m,v_m)-C(u,v)| \\&\leq&
\frac{4}{m} + \max_{0\leq j,k\leq m} \left|C_n\left(\frac{j}{m},\frac{k}{m}\right) - C\left(\frac{j}{m},\frac{k}{m}\right)\right| \\ &\leq& 
5\epsilon
\end{eqnarray*}
for all $n\in\mathds{N}$ large enough. 
\end{proof}
Consider also the paper of Sempi \cite{Sem} for further results in this area.

\begin{rem} \label{posStricRem} \rm
Let $\Sigma\in\mathds{R}^{2\times 2}$ be a positive-semidefinite matrix. Then, we obviously have $$N(0,\Sigma)\in\mathcal{M}^c \iff \sigma_{11}\sigma_{22}>0.$$ 
Assume this is the case, i.e. $\sigma_{11}\sigma_{22}>0$. Then
\begin{itemize}
\item[(i)]
$\Sigma$ is strictly positive definite, iff $\frac{|\sigma_{12}|}{\sqrt{\sigma_{11}\sigma_{22}}}<1$.
\item[(ii)]
$\mathcal{T}N(0,\Sigma)=W$, iff $\frac{\sigma_{12}}{\sqrt{\sigma_{11}\sigma_{22}}}=-1$.
\item[(iii)]
$\mathcal{T}N(0,\Sigma)=M$, iff $\frac{\sigma_{12}}{\sqrt{\sigma_{11}\sigma_{22}}}=1$.
\end{itemize}  
\end{rem}

\begin{prop} \label{GaussCopEqProp}
Let $\Sigma,\,\Sigma'\in\mathds{R}^{2\times 2}$ be two positive-definite matrices with $\sigma_{11}\sigma_{22}>0$. Then, it holds
\begin{equation} \label{GaussianCopRel}
\mathcal{T}N(0,\Sigma) = \mathcal{T} N(0,\Sigma')\quad \iff \quad \frac{\sigma_{12}}{\sqrt{\sigma_{11} \sigma_{22}}} = \frac{\sigma'_{12}}{\sqrt{\sigma'_{11} \sigma'_{22}}}.
\end{equation}
\end{prop}
\begin{proof}
We can assume because of the previous Remark \ref{posStricRem} w.l.o.g. that $\Sigma$ and $\Sigma'$ are strictly positive definite. Consider
$$\mathcal{T}N(0,\Sigma)(u,v) = \Phi_\Sigma(\phi_{\sigma_{11}}^{-1}(u),\phi_{\sigma_{22}}^{-1}(v)),\quad u,v\in[0,1]$$
where $\phi_\Sigma$ resp. $\phi_{\sigma_{jj}}$ denotes the cumulative distribution function of N$(0,\Sigma)$ resp. N$(0,\sigma_{jj})$, $j=1,2$. Set further $\phi\gl \phi_1$. Considering the respective densities, the equality of the left hand side in (\ref{GaussianCopRel}) is equivalent to
$$\frac{\partial^2}{\partial u\partial v} \mathcal{T} N(0,\Sigma) (u,v) = \frac{\partial^2}{\partial u\partial v} \mathcal{T}N(0,\Sigma')(u,v), \quad u,v\in [0,1].$$
Note that
\begin{eqnarray*}
\frac{\partial^2}{\partial u\partial v} \Phi_\Sigma(\phi_{\sigma_{11}}^{-1}(u),\phi_{\sigma_{22}}^{-1}(v)) &=&
\frac{\partial \phi_{\sigma_{11}}^{-1}}{\partial u}\frac{\partial \phi_{\sigma_{22}}^{-1}}{\partial v}\frac{\partial^2 \Phi_\Sigma}{\partial x\partial y} \\&=&
\left.\left(\frac{e^{-\frac{x^2}{2\sigma_{11}}-\frac{y^2}{2\sigma_{22}}}}{2\pi\sqrt{\sigma_{11}\sigma_{22}}} \right)^{-1} \frac{e^{-\frac{1}{2}(x,y)\Sigma^{-1} {x \choose y}}}{2\pi(\det\Sigma)^{\frac{1}{2}}}\right|_{{x \choose y} = {\sqrt{\sigma_{11}}\phi^{-1}(u) \choose \sqrt{\sigma_{22}}\phi^{-1}(v)}} \\&=&
\left.\sqrt{\frac{\sigma_{11}\sigma_{22}}{\det\Sigma}} e^{-\frac{1}{2}z^t(\Sigma^{-1}-D^{-2})z}\right|_{z = {x \choose y}},
\end{eqnarray*}
with
\begin{equation*}
D \gl \left(
\begin{array} {cc}
\sqrt{\sigma_{11}} & 0 \\
0 & \sqrt{\sigma_{22}}
\end{array}
\right).
\end{equation*}
We have
\begin{eqnarray*}
z^t(\Sigma^{-1} - D^{-2})z &=& \left(D{\phi^{-1}(u) \choose \phi^{-1}(v)}\right)^t (\Sigma^{-1} - D^{-2})D{\phi^{-1}(u) \choose \phi^{-1}(v)} \\&=&
{\phi^{-1}(u) \choose \phi^{-1}(v)}^t D^t(\Sigma^{-1}-D^{-2})D{\phi^{-1}(u) \choose \phi^{-1}(v)}.
\end{eqnarray*}
Furthermore, note
\begin{eqnarray*}
&& D^t(\Sigma^{-1}-D^{-2})D = D\Sigma^{-1}D - I = (D^{-1}\Sigma D^{-1})^{-1} - I \\&=&
\frac{1}{1-a^2}\left(
\begin{array}{cc}
1 & -a \\
-a & 1
\end{array}
\right) -I,\quad a\gl\frac{\sigma_{12}}{\sqrt{\sigma_{11}\sigma_{22}}}
\end{eqnarray*}
and
$$\frac{\det\Sigma}{\sigma_{11}\sigma_{22}} = 1-a^2.$$
This proves, together with the fact that
\begin{equation*}
\begin{array}{lcl}
(0,1)^2 &\to& \mathds{R}^2 \\
(u,v) &\mapsto& (\phi^{-1}(u),\phi^{-1}(v))
\end{array}
\end{equation*}
is a surjection, our claim.
\end{proof}

\begin{theorem} \label{CopGaussTheorem}
Let $F\in\mathcal{M}^c$ be a distribution with square integrable margins, i.e. 
$$\int (x^2 + y^2)\,dF(x,y) < \infty.$$
Then,
\begin{equation} \label{mainConvEq}
\sup_{0\leq u,v \leq 1} |\mathcal{T} \Psi(\lambda,F) - \mathcal{T} N(0,\Sigma)| \to 0, \quad \lambda\to\infty
\end{equation}
with
\begin{equation} \label{matrixEl}
\Sigma\gl \left(
\begin{array}{cc}
\int x^2\,dF(x,y) & \int xy\,dF(x,y) \\
\int xy\,dF(x,y) & \int y^2\,dF(x,y)
\end{array}
\right).
\end{equation}
\end{theorem}
\begin{proof}
Denote with $\sum_{j=1}^{N_t(\lambda)} X_j$ the corresponding compound Poisson process to $\Psi(\lambda,F)$, i.e.
$$P^{Z_\lambda} = e^{-\lambda} \sum_{k=0}^\infty\frac{\lambda^kF^{\ast k}}{k!}=\Psi(\lambda,F)$$
with 
$$Z_\lambda\gl\sum_{j=1}^{N_1(\lambda)}X_j,\quad \lambda>0.$$
Next, define a map
\begin{equation*}
\begin{array}{llcl}
\varphi_\lambda \,:\,&\mathds{R}^2 &\to& \mathds{R}^2 \\
&(x,y) &\mapsto& \frac{1}{\sqrt{\lambda}}((x,y) - EZ_\lambda).
\end{array}
\end{equation*}
It suffices to show that
\begin{equation} \label{weakConvEq}
\varphi_\lambda(Z_\lambda) \,\,{\buildrel d \over \to} \,\, N(0,\Sigma),\quad \lambda\to\infty
\end{equation}
because Theorem 2.4.3 in Nelsen \cite{Nel} yields
$$\mathcal{T}P^{\varphi_\lambda(Z_\lambda)} = \mathcal{T} P^{Z_\lambda} = \mathcal{T}\Psi(\lambda,F),$$
so that Proposition \ref{WeakCopConvPro} can be applied. Note that the latter application of $\mathcal{T}$ is allowed because of Lemma \ref{contLem} and the fact that $\varphi_\lambda$ is injective.

We verify (\ref{weakConvEq}) by use of L\'{e}vy's continuity theorem: For a convenient notation, set 
$$Y_v\gl \left<v,X_1\right>,\quad  v\in\mathds{R}^2$$
and observe
$$\mathcal{F} N(0,\Sigma) (v) = \exp\left(-\frac{v^t\Sigma v}{2}\right) = \exp\left(-\frac{EY_v^2}{2}\right).$$
Hence, it suffices to establish for every $v\in\mathds{R}^2$ the convergence
$$\mathcal{F}[Z_\lambda](v) \to \exp\left(-\frac{EY_v^2}{2}\right),\quad \lambda \to \infty.$$
Write for this
\begin{eqnarray*}
E\left(\exp\left(i\left<v,\varphi_\lambda(Z_\lambda)\right>\right)\right) &=& E\left(\exp\left(\frac{i}{\sqrt{\lambda}} (\left<v,Z_\lambda\right> - \lambda E\left<v,X_1\right>)\right)\right) \\&=&
E\left(\exp\left(\frac{i}{\sqrt{\lambda}} \sum_{j=1}^{N_1(\lambda)} \left<v,X_j\right>\right)\right) \exp(-i\sqrt{\lambda} E\left<v,X_1\right>) \\&=&
\exp\left(\lambda(\mathcal{F}[Y_v](\lambda^{-\frac{1}{2}})-1)\right) \exp(-i\sqrt{\lambda} EY_v) \\&=&
\exp\left(\lambda\left(\frac{i}{\sqrt{\lambda}} EY_v - \frac{1}{2\lambda} EY_v^2 + o(\lambda^{-1})\right)\right) \exp(-i\sqrt{\lambda} EY_v) \\&=&
\exp\left(-\frac{EY_v^2}{2} + o(1)\right),\quad \lambda\to\infty.
\end{eqnarray*}

This proves the desired convergence. Note that we used Sato \cite{Sat}[Theorem 4.3] for the third equal sign and Chow, Teicher \cite{ChoTei}[8.4 Theorem 1] for the fourth equal sign.

\end{proof}

\begin{cor} \label{copulaNotInjectivCor}
Let $F\in\mathcal{M}^c$ be a distribution with square integrable margins. Then, there exists another such distribution $G$ and a number $\Lambda > 0$ such that $\mathcal{T}F = \mathcal{T}G$, but
$$\mathcal{T}\Psi(\lambda,F) \neq \mathcal{T}\Psi(\lambda,G),\quad \lambda\geq\Lambda.$$
To be more precise, there exists $u_0,v_0\in(0,1)$ such that
$$\lim_{\lambda\uparrow \infty}\left|\mathcal{T}\Psi(\lambda,F) - \mathcal{T}\Psi(\lambda,G)\right|(u_0,v_0)>0.$$
Additionally, we can choose $G\in\mathcal{M}_+^c$ if $F\in\mathcal{M}_+^c$.
\end{cor}
\begin{proof}
For any $c,d\geq 0$ set $G\gl \delta_{(c,d)}\ast F\in\mathcal{M}^c$ and observe $G\in \mathcal{M}_+^c$ if $F\in\mathcal{M}_+^c$. Because of Theorem 2.4.3 in Nelsen \cite{Nel} we have $\mathcal{T}F = \mathcal{T}G$. Next, let $X$ be a random variable with $X\sim F$, so that $(X_1+c,X_2+d)\sim G$.
Due to Theorem \ref{CopGaussTheorem} together with Proposition \ref{GaussCopEqProp}, we only have to show that we can choose $c$ and $d$ such that $\rho(F) \neq \rho(G)$. For this purpose, note that
\begin{eqnarray*}
E(X_1+c)^2 &=& EX_1^2 + 2cEX_1 + c^2 \\
E(X_2+d)^2 &=& EX_2^2 + 2dEX_2 + d^2 \\
E(X_1+c)(X_2+d) &=& EX_1X_2 + cEX_2 + dEX_1 + cd.
\end{eqnarray*}
Set
\begin{equation*}
\begin{array}{rrcl}
\widetilde{\rho}: & \mathds{R}_+^2 &\to& [0,1] \\
&(c,d) &\mapsto& \left(\frac{E(X_1+c)(X_2+d)}{\sqrt{E(X_1+c)^2E(X_2+d)^2}}\right)^2
\end{array}
\end{equation*}
and observe
$$\widetilde{\rho}(c,d) = \frac{(EX_1X_2 + cEX_2 + dEX_1 + cd)^2}{(EX_1^2 + 2cEX_1 + c^2)(EX_2^2 + 2dEX_2 + d^2)}.$$
Assume $(c,d)\mapsto\widetilde{\rho}(c,d)$ is constant. Then
$$d\mapsto \widecheck{\rho}(d)\gl \lim_{c\uparrow\infty}\widetilde{\rho}(c,d) = \frac{(EX_2 + d)^2}{E(X_2+d)^2}$$
is also constant. This implies
$$\frac{(EX_2)^2}{EX_2^2} = \widecheck{\rho}(0) = \lim_{d\uparrow\infty}\widecheck{\rho}(d) = 1$$
which is only possible if $\text{Var}X_2 = 0$, i.e. $X_2$ is a.s. constant which is a contradiction to the assumed continuity of the second marginal distribution of $F$.
\end{proof}

\begin{rem} \label{copNotFunc} \rm
The equality
$$\mathcal{T}(\Psi(\lambda,F)) = \mathcal{T}(\Psi(\lambda,\mathcal{T}F)),\quad F\in\mathcal{M}_+^c,\quad \lambda>0$$
does not hold in view of Corollary \ref{copulaNotInjectivCor}. This implies that a map
$$\Psi_{\mathfrak{C}}\,:\,\mathfrak{C}\mapsto \mathfrak{C}$$
with the property
$$\Psi_{\mathfrak{C}}(C) = \mathcal{T}\Psi(\lambda,F),\quad C=\mathcal{T}F,\quad F\in\mathcal{M}_+^c$$
is not well-defined.
\end{rem}

\section{Two examples}
Consider the introduction of this paper for the motivation of the following two examples.

\begin{prop} \label{CopulaRhoPro}
We have 
$$\frac{1}{2}\leq\rho(C)\leq 1,\quad C\in\mathfrak{C}.$$ 
\end{prop}
\begin{proof}
The upper bound is a direct consequence of the Cauchy-Schwarz inequality. For the lower bound suppose $(U,V)\sim C$, i.e. in particular $U,V\sim U[0,1]$. We have to show that
$$EUV\geq\frac{1}{6}.$$
Cauchy-Schwarz yields
$$E(1-V)U \leq (E(1-V)^2EU^2)^\frac{1}{2} = EU^2.$$
This implies
$$EUV\geq EU - EU^2 = \frac{1}{2} - \frac{1}{3} = \frac{1}{6}$$
which proves the claim.
\end{proof}

\begin{exa} \label{ClaytonExample}
Let $(C_\theta)_{\theta\in[-1,\infty)\backslash\{0\}}$ be the family of Clayton copulas. Then, we have
$$\rho(\{C_\theta\,:\, \theta\in [-1,\infty)\backslash\{0\}\} \cup \{\Pi\}\cup \{M\}) = \left[\frac{1}{2},1\right].$$
\end{exa}
\begin{proof}
First, we show the continuity of the map 
\begin{equation} \label{LamConeq}
\theta\mapsto \rho(C_\theta),\quad \theta \in [-1,\infty)\backslash\{0\}.
\end{equation}
For this purpose, choose a sequence $(\theta_n)_{n\in\mathds{N}_0} \subseteq [-1,\infty)\backslash \{0\}$ with $\theta_n\to\theta_0$. The pointwise convergence
$$C_{\theta_n}(u,v) \to C_{\theta_0}(u,v),\quad 0\leq u,v\leq 1$$
yields the convergence of measures $C_{\theta_n}\,\, {\buildrel d \over \to}\,\, C_{\theta_0}$. Define the product function
$$H\,:\,[0,1]^2\to [0,1],\quad (u,v)\mapsto uv.$$
Since $H$ is continuous, we have $C_{\theta_n}^H\,\, {\buildrel d \over \to}\,\, C_{\theta_0}^H$, which implies
\begin{equation} \label{copWeakEq}
C_{\theta_n}(H\leq t) \to C_\theta (H\leq t), \quad \theta_n \to \theta\quad (t\text{-a.e.}).
\end{equation}
Finally we can write
\begin{eqnarray*}
\left|\int H\,dC_{\theta_n} - \int H\,dC_\theta\right| &=& \left|\int_0^1 C_{\theta_n}(H>t)\,dt - \int_0^1 C_\theta(H>t)\,dt\right| \\ &\leq&
\int_0^1|C_{\theta_n}(H\leq t) - C_\theta(H\leq t)|\,dt \to 0,\quad n\to\infty
\end{eqnarray*}
where the last convergence holds because of (\ref{copWeakEq}) and dominated convergence. This proves the claimed continuity. Next, observe the pointwise convergences
\begin{eqnarray*}
&&C_{\theta_n} \to \Pi,\quad \theta_n\to 0,\quad \theta_n \in [-1,\infty)\backslash\{0\}, \\
&&C_{\theta_n} \to M,\quad \theta_n \to\infty
\end{eqnarray*}
and $C_{-1} = W$, cf. Nelsen \cite{Nel} (4.2.1). This completes, together with the continuity of (\ref{LamConeq}), 
$$\rho(W)=\frac{1}{2},\quad \rho(\Pi) = \frac{3}{4},\quad \rho(M)=1$$
and Proposition \ref{CopulaRhoPro},
the proof.
\end{proof}

\begin{exa} \label{remainderExample}
Let $\{U_k\,:\,k\in\mathds{N}\}\cup\{V_k\,:\,k\in\mathds{N}\}$ be a family of i.i.d $U[0,1]$ distributed random variables and fix any $0<\epsilon<1$. Set
$$T_\epsilon\gl \inf\{k\in\mathds{N}\,:\,|U_k-V_k|\geq \epsilon\}.$$
Then, the following two statements are true:
\begin{itemize}
\item[(i)]
$(U_{T_\epsilon},V_{T_\epsilon})\sim U(I_\epsilon)$ with $I_\epsilon=\{(u,v)\in [0,1]^2\,:\,|u-v|\geq \epsilon\}$.
\item[(ii)]
Set
\begin{equation*}
\begin{array}{rrcl}
\varphi: & (0,1) &\to& [0,1] \\
&\epsilon &\mapsto& \rho(P^{(U_{T_\epsilon},V_{T_\epsilon})}).
\end{array}
\end{equation*}
It holds $(0,\frac{3}{4})\subseteq \varphi((0,1))$.
\end{itemize}
\end{exa}
\begin{proof}
To have an unambiguous notation in this proof, the two dimensional Lebesgue measure is denoted in the following with $\mathfrak{l}^2$ instead of $\lambda^2$.

(i): Let $A\in\mathcal{B}^2$ and write
\begin{eqnarray*}
&&P((U_{T_\epsilon},V_{T_\epsilon})\in A)  \\
&=& \sum_{k=1}^\infty P((U_{T_\epsilon},V_{T_\epsilon})\in A|T_\epsilon = k) P(T_\epsilon=k) \\&=&
\sum_{k=1}^\infty  P((U_k,V_k)\in A|(U_1,V_1)\in I_\epsilon^c,\ldots,(U_{k-1},V_{k-1})\in I_\epsilon^c,(U_k,V_k)\in I_\epsilon) P(T_\epsilon =k)\\&=&
\sum_{k=1}^\infty P((U_k,V_k)\in A|(U_k,V_k)\in I_\epsilon) P(T_\epsilon = k)\\&=&
\frac{P((U_1,V_1)\in A\cap I_\epsilon)}{P((U_1,V_1)\in I_\epsilon)}.
\end{eqnarray*}
Note that $P(T_\epsilon=\infty)=0$. \\
(ii): Obviously $\mathfrak{l}^2(I_\epsilon) = (1-\epsilon)^2$ holds.
Next, we obtain
\begin{eqnarray*}
EU_{T_\epsilon}^2 &=& (\mathfrak{l}^2(I_\epsilon))^{-1}\int_{[0,1]^2} u^2 \one_{I_\epsilon}(u,v) d\mathfrak{l}^2(u,v) \\ &=&
(1-\epsilon)^{-2} \left(\int_\epsilon^1\int_{0}^{u-\epsilon} u^2\,dv\,du + \int_\epsilon^1\int_{0}^{v-\epsilon} u^2\,du\,dv\right) \\&=&
\frac{\epsilon^2}{6}+\frac{1}{3}
\end{eqnarray*}
and
\begin{eqnarray*}
E(U_{T_\epsilon}V_{T_\epsilon}) &=& (\mathfrak{l}^2(I_\epsilon))^{-1} \int_{[0,1]^2}uv\one_{I_\epsilon}(u,v)\,d\mathfrak{l}^2(u,v) \\&=&
2(1-\epsilon)^{-2}\int_\epsilon^1\int_{0}^{u-\epsilon} uv\,dv\,du \\&=&
\frac{(1-\epsilon)(3+\epsilon)}{12}.
\end{eqnarray*}
A symmetry argument yields $E(V_{T_\epsilon}^2) = E(U_{T_\epsilon}^2)$, so that we have
$$\varphi(\epsilon) = \frac{(1-\epsilon)(3+\epsilon)}{2(\epsilon^2+2)}.$$
Continuity of $\varphi$ and
$$0 = \lim_{\epsilon\uparrow 1}\varphi(\epsilon) < \lim_{\epsilon\downarrow 0}\varphi(\epsilon) = \frac{3}{4}$$
proves (ii).
\end{proof}

\section{Simulation study}
In this section, we investigate the convergence (\ref{mainConvEq}) in Theorem \ref{CopGaussTheorem} by means of numerical MATLAB simulations. For this purpose, choose
$$F = F_{\theta} = C_{\theta},\quad \theta = 0,1,2,5,$$
i.e. the jump distribution is a two dimensional Clayton Copula, compare Section \ref{defSec}. Since we need for the convergence (\ref{mainConvEq}) the values of the matrix (\ref{matrixEl}), consider the following proposition.
\begin{prop} \label{EUVProp}
Let $C$ be a two dimensional Copula with existing first partial derivatives. Further, let $U,V$ be to random variables, such that $P^{(U,V)}=C$. Then, we have
\begin{itemize}
\item[(i)]
$P(V\leq v| U=u) = \frac{\partial C}{\partial u}(u,v),\quad 0\leq u,v \leq 1,$
\item[(ii)]
$E(UV)=\frac{1}{2} - \int_0^1\int_0^1 u \frac{\partial C}{\partial u}(u,v)\,dv\,du.$
\end{itemize}
\end{prop}  
\begin{proof}
Fix any $0\leq v\leq 1$. Then, it holds for every $0\leq a \leq 1$
$$\int_0^a P(V\leq v|U=u)\, P^U(du) = \int_{\{U\leq a\}} E(\one_{\{V\leq v\}} |U)\,dP = P(U\leq a,\, V\leq v)$$
and
$$\int_0^a \frac{\partial C}{\partial u} (u,v)\,P^U(du) = \int_0^a\frac{\partial C}{\partial u}(u,v)\,du = C(a,v) = P(U\leq a,\, V\leq v)$$
which proves (i). Concerning (ii) consider
\begin{equation} \label{EUVeq1}
E(UV)=\int U E(V|U)\,dP = \int u E(V|U=u)\,P^U(du) = \int_0^1 u E(V|U=u)\, du.
\end{equation}
Further, it follows from Bauer \cite{Bau}[Theorem 23.8]
\begin{equation} \label{EUVeq2}
E(V|U=u) = \int_0^1 P^{V|U=u}((v,1])\,dv = \int_0^1 (1- P(V\leq v|U=u))\,dv.
\end{equation}
Finally, a substitution of (\ref{EUVeq2}) in (\ref{EUVeq1}) proves together with (i) the claim (ii).
\end{proof}
We have for $\theta>0$
$$C_\theta(u,v) = (u^{-\theta} + v^{-\theta} - 1)^{-\frac{1}{\theta}},\quad 0\leq u,v\leq 1.$$
This yields
$$\frac{\partial C_\theta}{\partial u} = u^{-\theta-1}(u^{-\theta} + v^{-\theta} - 1)^{-\frac{1+\theta}{\theta}},\quad 0\leq u,v\leq 1$$
and, thus, we have together with Proposition \ref{EUVProp} (ii)
\begin{equation} \label{EUVClaytonEq}
E_\theta(UV) = \frac{1}{2} - \int_0^1\int_0^1 u^{-\theta} (u^{-\theta} + v^{-\theta} - 1)^{-\frac{1+\theta}{\theta}}\,dv\,du,\quad \theta >0.
\end{equation}

\begin{table}[htbp]
\begin{equation*}
\begin{array}{c|c|c|c|c}
\theta & 0 & 1 & 2 & 5 \\
\hline
\rho(C_\theta) & 0.7500 & 0.8696 & 0.9206 & 0.9712
\end{array}.
\end{equation*}
\caption{Gaussian limit copulas as a function of $\theta$}
\label{ClaytonGaussianTable}
\end{table}

Note the connection between $E(UV)$ and Spearman's rho $\rho_S$, cf. Nelsen \cite{Nel}[Theorem 5.1.6.], i.e. we have $\rho_S = 12E(UV)-3$. However, a closed form expression of Spearman's rho for Clayton copulas is not known and hence, we perform numerical calculations of the integrals in (\ref{EUVClaytonEq}) and state the results in Table \ref{ClaytonGaussianTable}. Observe further, that it holds $\int x^2 dF_{\theta} = \int_0^1 x^2\,dx = \frac{1}{3}$, $\theta\geq 0$ and that $\theta=0$ implies that $U$ and $V$ are independent, i.e. $E_0(UV)=E_0(U)E_0(V)=\frac{1}{4}$.

We are able to simulate i.i.d. samples of a Clayton Copula $C_\theta$, $\theta>0$ by use of the following proposition.
\begin{prop}
Let $U$ and $Z$ be independent $U[0,1]$ distributed random variables. Set
$$V\gl \left(1+U^{-\theta}\left(Z^{-\frac{\theta}{1+\theta}}-1\right)\right)^{-\frac{1}{\theta}}.$$
Then, we have
$$(U,V)\,\, {\buildrel \text{d} \over =}\,\, C_\theta,\quad \theta > 0.$$
\end{prop}
\begin{proof}
Proposition \ref{EUVProp} (i) yields
$$F_{\theta,u}(v)\gl P_{\theta}(V\leq v|U=u) = \frac{\partial C}{\partial u}(u,v),\quad\theta>0,\quad 0\leq u,v\leq 1.$$
Observe that
\begin{equation} \label{UVdistrEq}
F_{\theta,U}^{-1}(Z)\,\, {\buildrel \text{d} \over =}\,\, (U,V),\quad \theta > 0
\end{equation}
is equivalent to
\begin{equation} \label{UVdistrEq2}
F_{\theta,u}^{-1}(Z)\,\, {\buildrel \text{d} \over =}\,\, P_{\theta}^{V|U=u},\quad \theta>0, \quad 0\leq u\leq 1
\end{equation}
which is the condition of (\ref{UVdistrEq}) on $U=u$ and that (\ref{UVdistrEq2}) is a known fact, i.e. (\ref{UVdistrEq}) is true. Thus, it remains to calculate $F_{\theta,u}^{-1}$, $0\leq u\leq 1$, $\theta>0$. Consider for this
$$\frac{\partial C_{\theta}}{\partial u} = u^{-\theta-1}(u^{-\theta} + v^{-\theta} -1)^{-\frac{1+\theta}{\theta}} = F_{\theta,u}(v).$$
A straightforward calculation yields
$$F_{\theta,u}^{-1}(y) = \left(1+u^{-\theta}\left(y^{-\frac{\theta}{1+\theta}}-1\right)\right)^{-\frac{1}{\theta}},\quad 0\leq u,y\leq 1.$$
This proves together with (\ref{UVdistrEq}) the claim.
\end{proof}

Consider in this context also Lee \cite{Lee}. The next proposition is useful for the simulation of the Gaussian limit copulas in Theorem \ref{CopGaussTheorem}.
\begin{prop}
Fix any $-1\leq \tau \leq 1$ and let $X$ and $Y$ be two independent, N(0,1) distributed random variables. Set
\begin{equation*}
U\gl \phi(X),\quad V\gl \phi(\tau X + \sqrt{1-\tau^2}\, Y),\quad
\Sigma \gl \left(
\begin{array} {cc}
1 & \tau \\
\tau & 1
\end{array}
\right)
\end{equation*}

where $\phi$ denotes the cumulative distribution function of N(0,1). Then, it holds
$$(U,V)\,\, {\buildrel \text{d} \over =}\,\, \mathcal{T} N(0,\Sigma),\quad \rho(N(0,\Sigma)) = \tau.$$
\end{prop}
\begin{proof}
Observe that
$\left(\begin{array} {cc}
1 & 0 \\
\tau & \sqrt{1-\tau^2}
\end{array}
\right)$
is the Cholesky decomposition of
$\left(\begin{array} {cc}
1 & \tau \\
\tau & 1
\end{array}
\right)$, i.e.
$$(\phi^{-1}(U),\phi^{-1}(V)) \,\,{\buildrel \text{d} \over =}\,\, N(0,\Sigma).$$
This proves together with the definition of $\rho$, cf. (\ref{rhodef}), the claim.
\end{proof}

Figure \ref{jmpFinalFig} shows a scatter diagram of respectively $500$ i.i.d. samples of Clayton copulas and their associated Gaussian limit copulas, compare Table \ref{ClaytonGaussianTable}.

\begin{figure}[htbp] 
	\includegraphics[clip=true, bb=100 50 870 500, width=15.5cm]{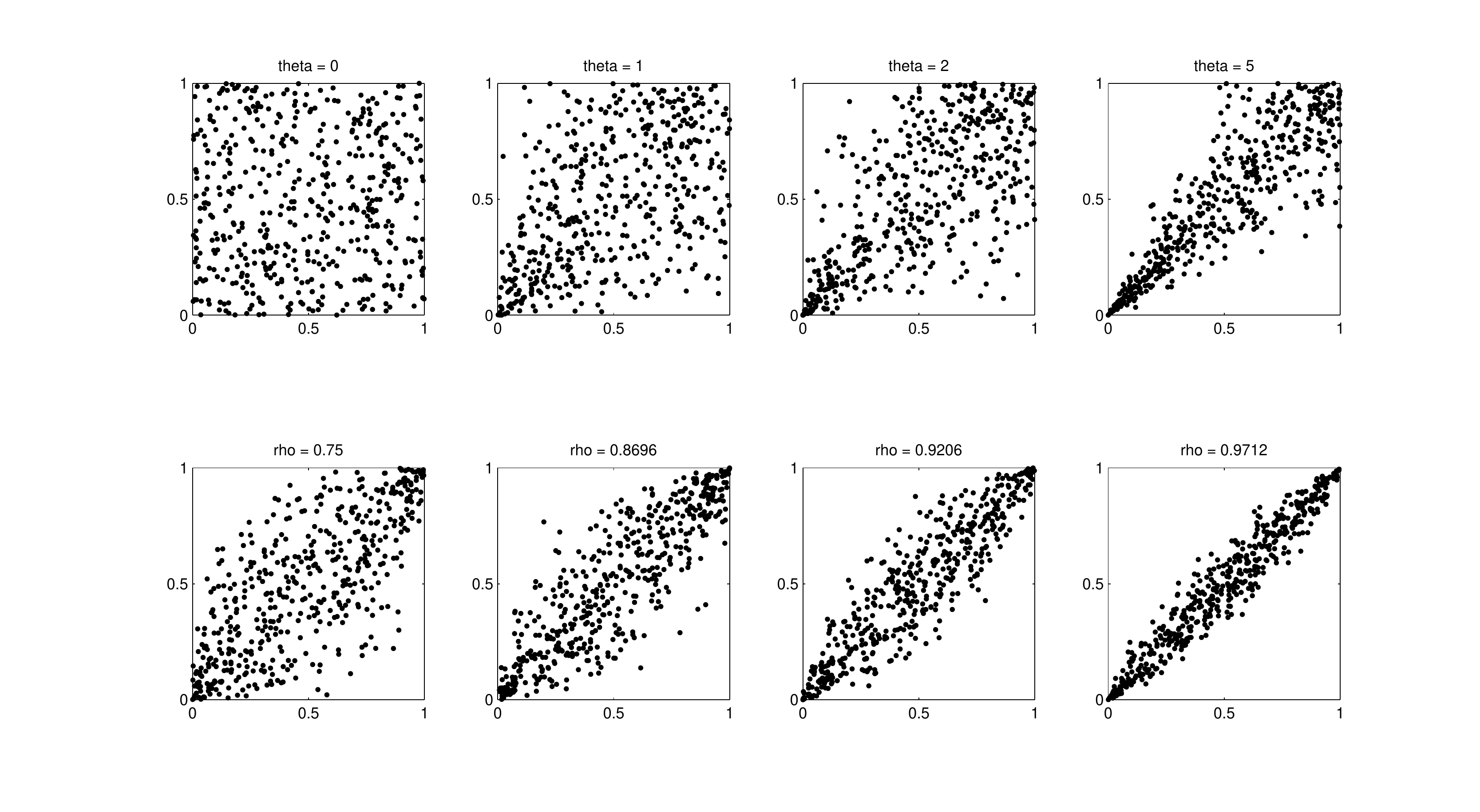}
\caption{Clayton copulas and their associated Gaussian limit copulas}
\label{jmpFinalFig}
\end{figure}

\newpage
Figure \ref{directConvFig} simulates the copulas $\mathcal{T}\Psi(\lambda,F_5)$, $\lambda=3,5,7,20$ by plotting respectively $500$ i.i.d. samples. Considering Theorem \ref{CopGaussTheorem} and Table \ref{ClaytonGaussianTable}, it holds
$$\mathcal{T} \Psi(\lambda,F_5) \to \mathcal{T}\Psi(\infty,F_5)\gl\mathcal{T}N\left(0,
\left(\begin{array} {cc}
1 & \tau \\
\tau & 1
\end{array}
\right)
\right),\quad \lambda\to\infty,$$ 
uniformly on $[0,1]^2$ with $\tau=0.9712$. However, Figure \ref{directConvFig} indicates that even $\mathcal{T}\Psi(3,F_5)$ is close to $\mathcal{T}\Psi(\infty,F_5)$. All four copula plots in Figure \ref{directConvFig} resemble the associated Gaussian limit copula $\mathcal{T}\Psi(\infty,F_5)$ in Figure \ref{jmpFinalFig}. Hence, this graphical comparison method seems to be inadequate, i.e. not subtle enough.

\begin{figure}[htbp] 
	\includegraphics[clip=true, bb=100 180 870 370, width=15.5cm]{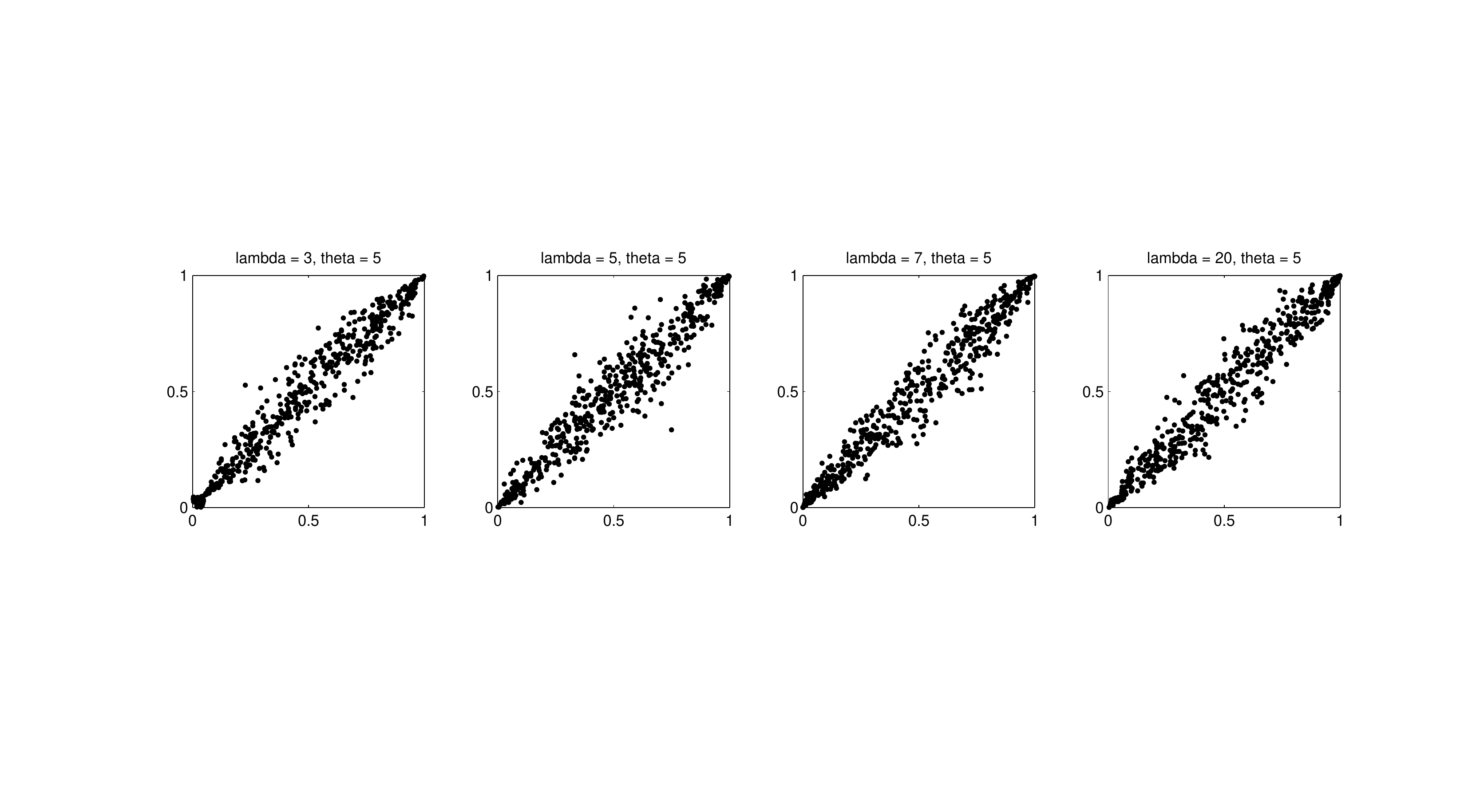}
\caption{Compound Poisson copulas for increasing $\lambda$.}
\label{directConvFig}
\end{figure}

As a solution, we simulate in Figure \ref{diffThetaFig} the density of the difference
$$\Delta(\lambda,\theta)\gl |\mathcal{T}\Psi(\lambda,F_{\theta}) - \mathcal{T}\Psi(\infty,F_{\theta})|,\quad \lambda=3,5,7,20,\,\,\theta=0,1,2,5.$$
To be more precise, let $(X_{\lambda,\theta}^k)_{k=1,\ldots,N}$ resp. $(Y_{\infty,\theta}^k)_{k=1,\ldots,N}$ be $N$ i.i.d. samples following the distribution $\mathcal{T} \Psi(\lambda,F_{\theta})$ resp. $\mathcal{T}\Psi(\infty,F_{\theta})$ and let $K\subseteq [0,1]^2$ be a measurable set. Then, it holds by the strong law of large numbers
\begin{eqnarray}
\left| \frac{1}{N} \sum_{k=1}^N (\one_K(X_{\lambda,\theta}^k) - \one_K(Y_{\infty,\theta}^k))\right| &=&
|E(\one_K(X_{\lambda,\theta}^1)) - E(\one_K(Y_{\infty,\theta}^1))| + o_P(1) \label{DEst}\\ &=&
|(\mathcal{T}\Psi(\lambda,F_{\theta}))(K) - (\mathcal{T}(\Psi(\infty,\theta))(K)| + o_P(1) \nonumber \\ &=&
\Delta(\lambda,\theta)(K) + o_P(1). \nonumber
\end{eqnarray}
Set
$$K_{i,j} \gl \left[\frac{i}{M},\frac{i+1}{M}\right) \times \left[\frac{j}{M},\frac{j+1}{M}\right),\quad i,j = 0,\ldots,M-1$$
for some fixed $M\in\mathds{N}$. Thus, (\ref{DEst}) yields that the numbers
$$c_{i,j}\gl \left|\frac{M^2}{N}\sum_{k=1}^N(\one_{K_{i,j}}(X_{\lambda,\theta}^k) - \one_{K_{i,j}}(Y_{\infty,\theta}^k))\right|,\quad i,j=0,\ldots,M$$
are approximations of 
$$\delta(\lambda,\theta)\gl\frac{\Delta(\lambda,\theta)(K_{i,j})}{\lambda^2(K_{i,j})},\quad i,j=0,\ldots,M-1,$$
which is for large $M$ an approximation to the density of $\Delta(\lambda,\theta)$ in the point $(\frac{i}{M},\frac{j}{M})$. We set $N=10^6$ and $M=30$ in our simulations in Figure \ref{diffThetaFig} . 

It remains to find a suitable representation of the function
$$\Psi\,:\,\{0,\ldots,29\}^2 \to \mathds{R}_+,\quad (i,j)\mapsto c_{i,j}.$$
$\Psi$ is realized as a $2$D-plot in the following way: Plot in the square $K_{i,j}$, $i,j=0,\ldots,29$ randomly $\lfloor \alpha c_{i,j}\rfloor$ points where $\alpha>1$ is a scaling constant in order to controll the average point intensity of the respective plots in Figure \ref{diffThetaFig}. Here, we choose $\alpha=20$, so that $m_{i,j}$ plotted points in a square $K_{i,j}$ represents a density difference of $\frac{m_{i,j}}{20}$.

The plots in Figure \ref{diffThetaFig} clearly illustrate the decrease of the density difference since the respective scatter diagrams are getting thinner with increasing $\lambda$. Further, the largest difference is around the origin. This is not surprising since a compound Poisson process does not jump with the probability $e^{-\lambda}$ in the unit intervall and thus, the corresponding distribution has an atom at the origin. This is in utter contrast to the continuity of the Gaussian limit copulas. However, this effect vanishes for increasing $\lambda$, compare Table \ref{lambdaTable}.

\begin{figure}[htbp] 
	\includegraphics[clip=true, bb=100 150 870 400, width=15.5cm]{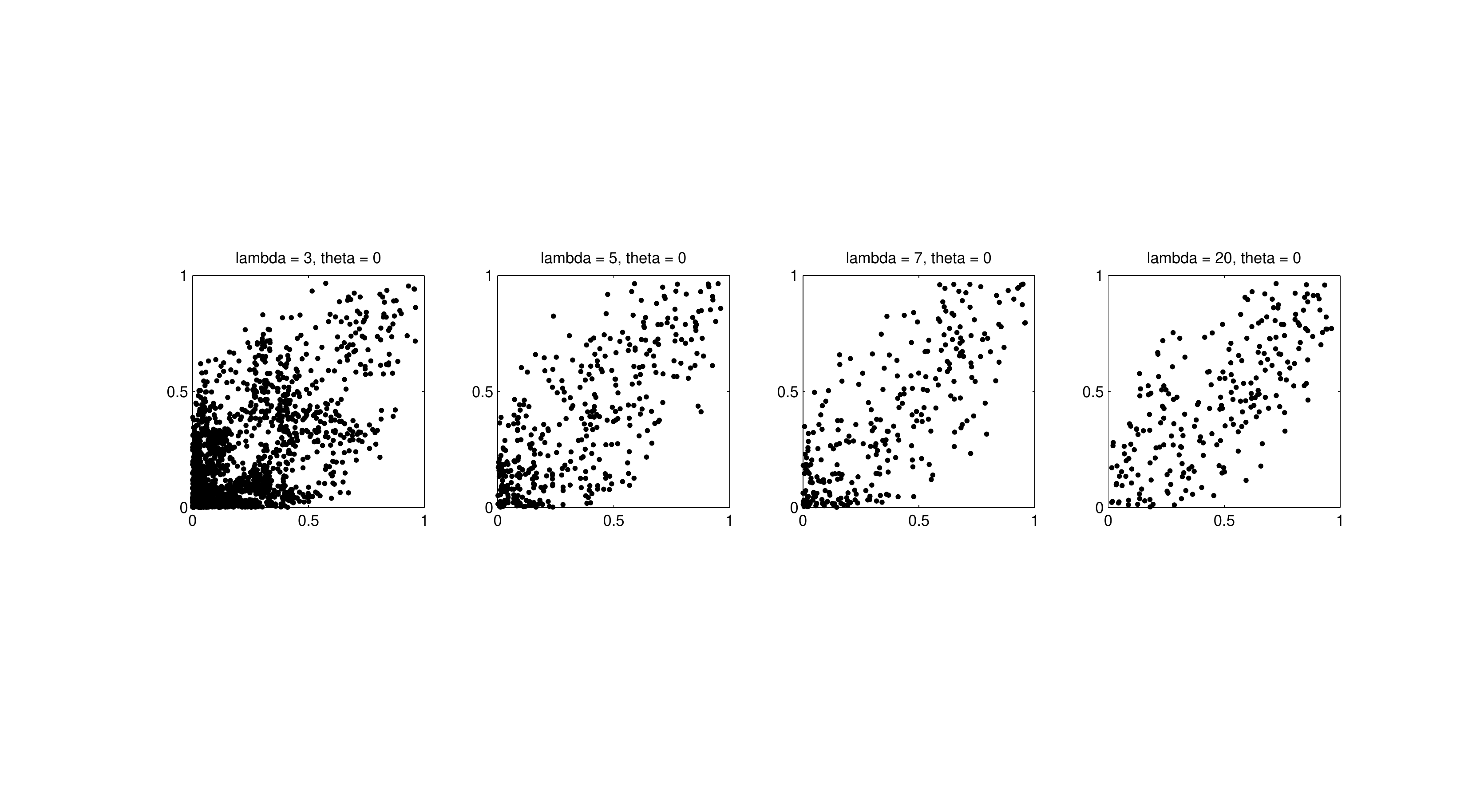}
	\includegraphics[clip=true, bb=100 150 870 400, width=15.5cm]{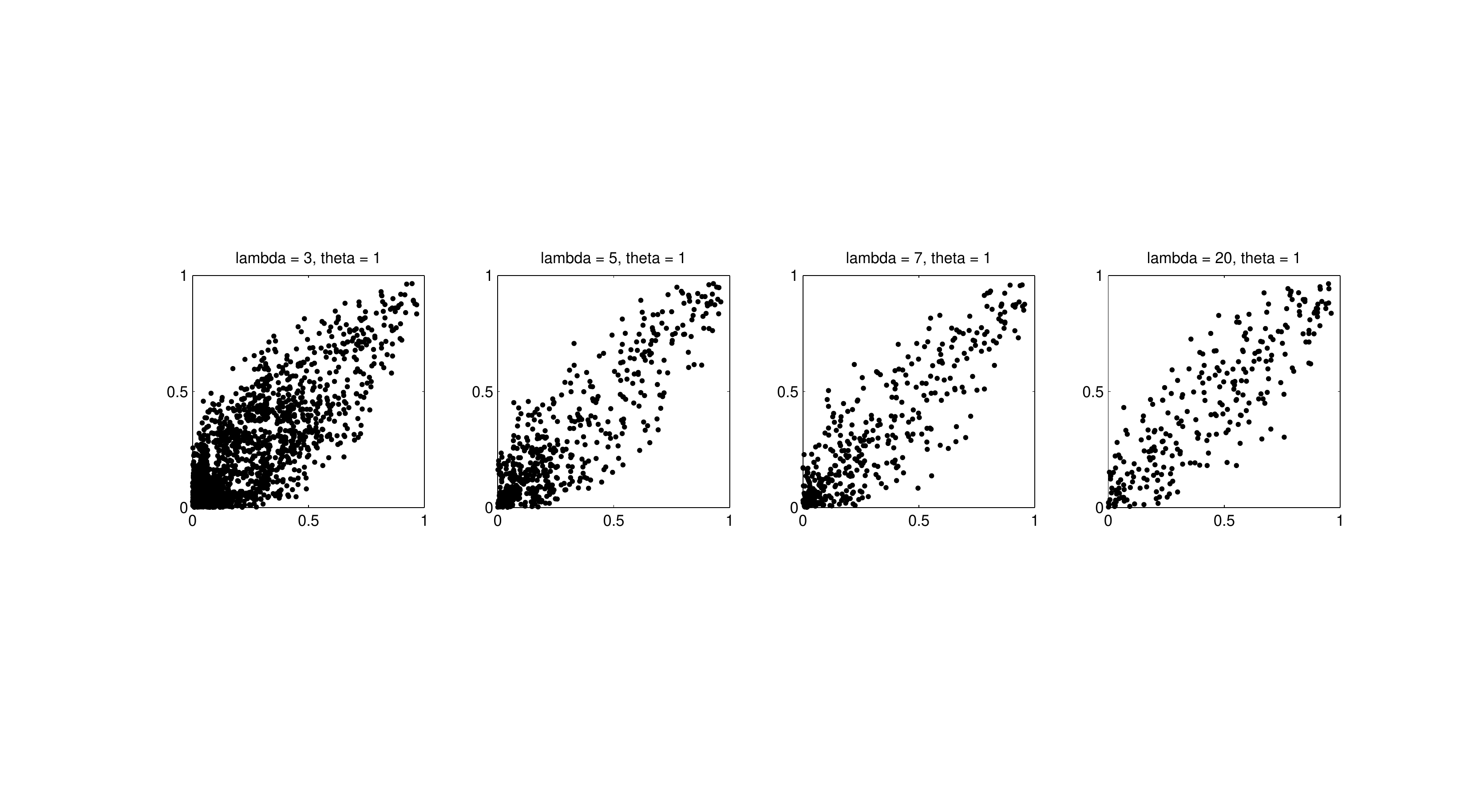} 
	\includegraphics[clip=true, bb=100 150 870 400, width=15.5cm]{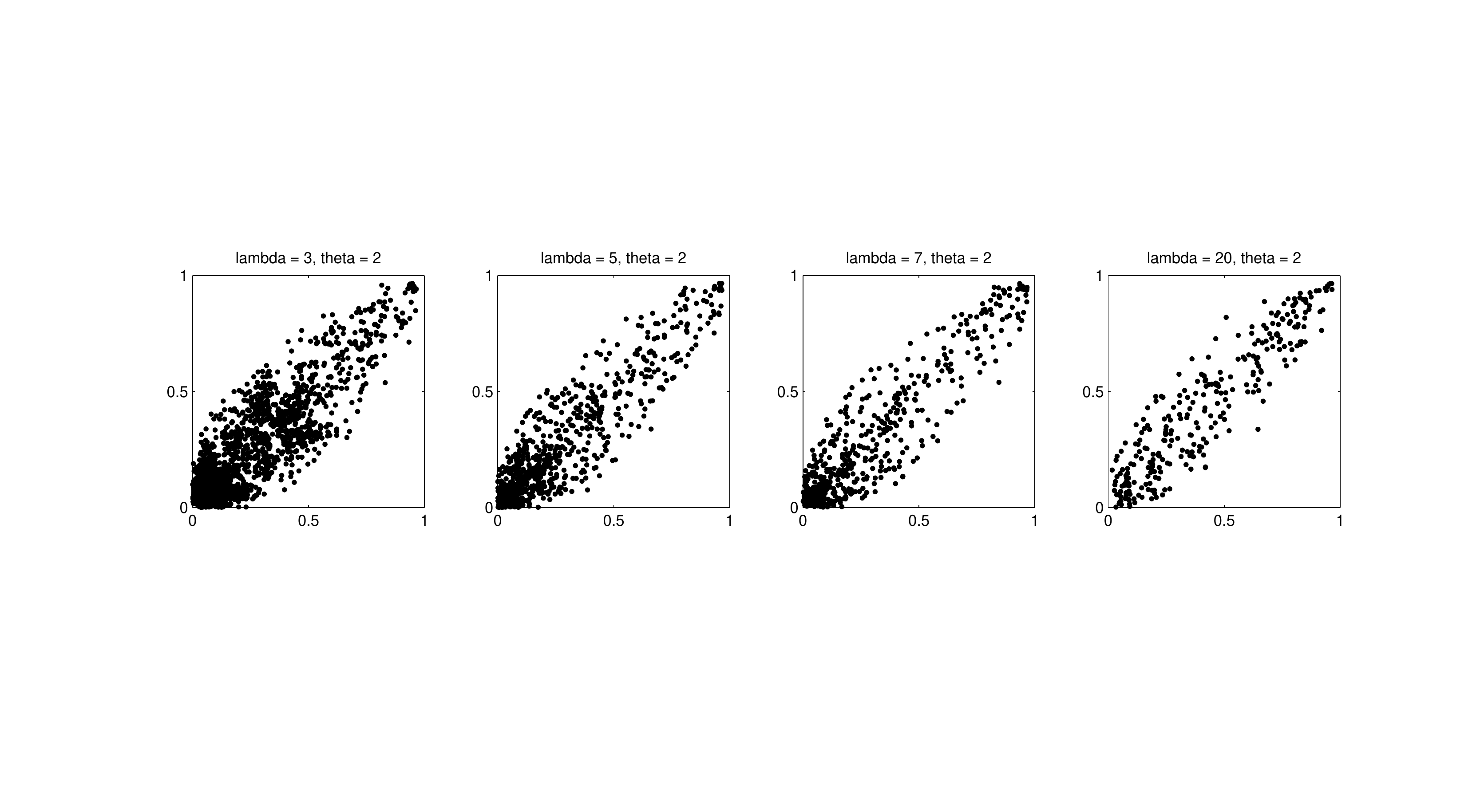}
	\includegraphics[clip=true, bb=100 150 870 400, width=15.5cm]{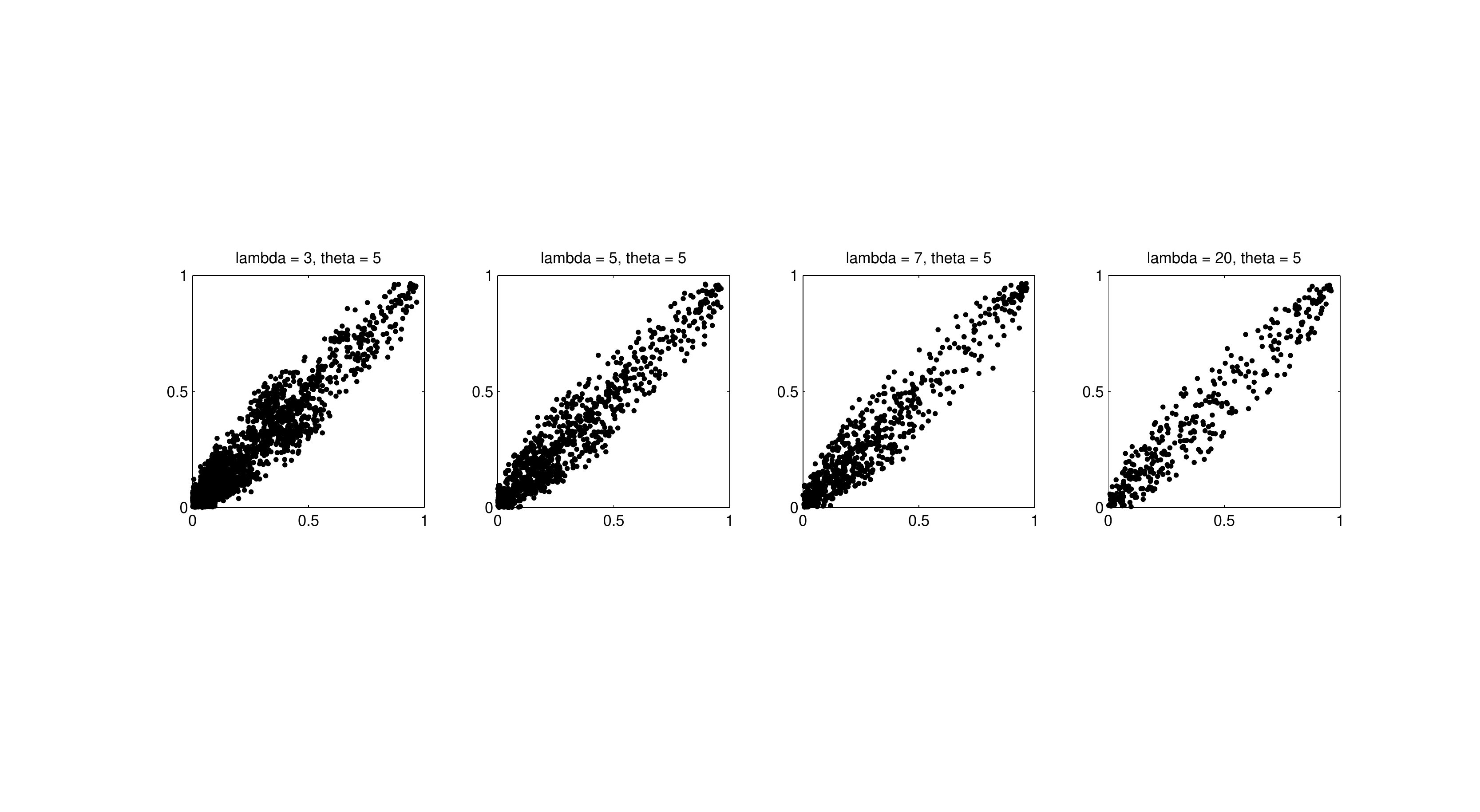}
\caption{Density difference plots}
\label{diffThetaFig}
\end{figure}


\begin{table}[htbp]
\begin{equation*}
\begin{array}{c|c|c|c|c}
\lambda & 3 & 5 & 7 & 20 \\
\hline
e^{-\lambda} & 0.4978\cdot 10^{-1} & 0.6737\cdot 10^{-2} & 0.91188\cdot 10^{-3} & 0.2061\cdot 10^{-8}
\end{array}
\end{equation*}
\caption{Probability of no jumps}
\label{lambdaTable}
\end{table}

Finally, Table \ref{diffMassTable} states the total difference mass of the limit copula and its approximation, i.e.
$$\|\Delta(\lambda,\theta)\|\gl \Delta(\lambda,\theta)([0,1]^2) = \sum_{i,j=0}^{M-1} \delta(\lambda,\theta)(K_{i,j})\,\lambda^2(K_{i,j}) = \frac{\lambda^2(K_{0,0})}{\alpha} \sum_{i,j=0}^{M-1} \alpha c_{i,j} = \frac{\text{number of dots}}{18000}.$$
The entries in Table \ref{diffMassTable} are decreasing in $\lambda$ and increasing in $\theta$. This can be interpreted in the sense that in the case of a Clayton copula, a stronger dependence of the components in the jump distribution results in a slower convergence to the Gaussian limit copula.

\begin{table}[htbp]
\begin{equation*}
\begin{array}{c|cccc}
\theta \backslash \lambda & 3 & 5 & 7 & 20 \\
\hline
0 & 0.1058 & 0.0227 & 0.0161 & 0.0132 \\
1 & 0.1095 & 0.0311 & 0.0228 & 0.0151 \\
2 & 0.1115 & 0.0413 & 0.0283 & 0.0161 \\
5 & 0.1293 & 0.0591 & 0.0411 & 0.0201
\end{array}
\end{equation*}
\caption{Total difference mass $\|\Delta(\lambda,\theta)\|$}
\label{diffMassTable}
\end{table}

\newpage

{\bf Acknowledgements.} The financial support of the Deutsche Forschunsgemeinschaft (FOR 916, project B4) is gratefully acknowledged.

\end{document}